\documentclass[12pt]{article}
\usepackage[utf8]{inputenc}
\usepackage[english]{babel}
\usepackage[dvips]{graphicx}
\usepackage{indentfirst}
\usepackage{latexsym}
\usepackage{amsthm}
\usepackage{amsfonts}
\usepackage{amssymb}
\usepackage{amsmath}
\usepackage{hyperref}

\theoremstyle{plain}
\newtheorem{theorem}{Theorem}

\newtheorem{corollary}{Corollary}
\newtheorem{proposition}{Proposition}
\newtheorem{assertion}{Claim}
\theoremstyle{definition}

\newtheorem{property}{Property}

\begin{document}

\noindent UDC 519.17

\title{Digraphs of potential barriers: properties of their tree structure and algorithm for constructing minimum spanning forests} 
\author{V.\,A. Buslov} 
\begin{center}
{\bf Digraphs of potential barriers: properties of their tree structure and algorithm for constructing minimum spanning forests} 
\end{center}
\begin{center}
{\large V.\,A. Buslov}
\end{center}

\begin{abstract}
  For a weighted digraph without loops $V$, the arc weights of which can be obtained from an undirected graph with loops ${\sf P}$ according to the rule $v_{ij}=p_{ij}-p_{ii}$, the properties are studied. An effective algorithm for constructing directed trees of minimum weight and an algorithm for constructing spanning directed forests of minimum weight are proposed.  
\end{abstract}

This article is based on the results of works \cite{V6}-\cite{V8}. Definitions and designations correspond to those adopted there. 

\section{Definitions and notations}

The work uses both directed and undirected graphs, forests and trees. For brevity, we use the terms graph, forest and tree, denoting both undirected and directed, if this does not lead to misunderstandings. Where necessary, we indicate exactly what type of graphs we have in mind.

For a digraph $G$, we denote the set of its vertices by ${\cal V}G$, and the set of arcs by ${\cal A}G$. If the graph $G$ is undirected, then ${\cal E}G$ is the set of its edges.

The original object is a weighted digraph $V$, with a set of vertices ${\cal V}V={\cal N}$, $|{\cal N}|=N$; its arcs $(i,j)\in{\cal A}V$ are assigned real weights $v_{ij}$. We consider spanning subgraphs (with a set of vertices ${\cal N}$) that are entering (incoming) forests. An entering forest is a digraph in which no more than one arc comes from each vertex and there are no contours. A tree is a connected component of a forest. The root of a tree (forest) is the vertex from which arc does not come. Let $T^F_i$ denote the tree of  forest $F$ rooted at vertex $i$. We denote the set of roots of the forest $F$ by ${\cal K}_F$.  The outgoing forest is obtained by replacing all the arcs of the entering forest with inverse ones (then the root is the vertex into which arcs do not go).  

Of the directed ones, only entering forests are used, in the future simply --- forests, until the moment when undirected graphs do not appear in the presentation. An undirected forest is a graph without cycles, its connected components are trees.

The subgraph $H$ of a graph $G$ induced by the set ${\cal S}\subset{\cal V}G$ (or the restriction of the graph $G$ to the set ${\cal S}$) is a subgraph in which ${\cal V}H={\cal S}$, and the set of its arcs is all arcs of the graph $G$, both ends of which belong to the set ${\cal S}$. For it we use the notation $H=G|_{\cal S}$.  

If there is an arc whose outcome belongs to the set ${\cal S}$, but its entry does not, then we say that the arc comes from the set ${\cal S}$. Similarly, if there is an arc whose entry belongs to ${\cal S}$ and whose outcome does not belong, then we say that the arc enters ${\cal S}$.

The outgoing neighborhood ${\cal N}^{out}_{\cal S}(G)$ of the set ${\cal S}$ is the set of entries of arcs coming in graph $G$ from the set ${\cal S}$; The incoming neighborhood ${\cal N}^{in}_{\cal S}(G)$ is defined similarly.

For a subgraph $G$ of a graph $V$ and a subset of the vertex set ${\cal S}\subseteq {\cal N}$, we introduce weights   
\begin{equation}
\Upsilon^G_{\cal S}=\sum_{\begin{smallmatrix}i\in{\cal S} \\ (i,j)\in {\cal A}G\end{smallmatrix}} v_{ij} \ , \ \ \Upsilon^G=\Upsilon^G_{\cal N}=\sum_{(i,j)\in {\cal A}G} v_{ij} \ . \label{ves}
\end{equation}
The value $\Upsilon^G_{\cal S}$ is also formed from arcs whose entries do not belong to the set ${\cal S}$. If in a graph $G$ no arcs come from the set ${\cal S}$ itself, then $\Upsilon^G_{\cal S}=\Upsilon^{G|_{\cal S}} $. 

${\cal F}^k$ is a set of forests consisting of $k=1,2,\ldots ,N$ trees.  We denote the minimum weight $\Upsilon^F$ among all forests $F\in{\cal F}^k$ by $\phi^k$:
\begin{equation}
\phi^k=\min_{F\in{\cal
F}^k}\Upsilon^F \ .
\label{phi}
\end{equation}

If ${\cal F}^k=\emptyset$, we set $\phi^k=\infty$, in particular, $\phi^0=\infty$. The set ${\cal F}^N$ consists of only one empty forest and $\phi^N=0$. Any forest in ${\cal F}^{N-1}$ contains exactly one arc, so $\phi^{N-1}=\min\limits_{(i,j)\in{\cal A}V}v_{ij}$. 

$\tilde{\cal F}^k$ is a subset of the set of forests ${\cal F}^k$ on which the minimum of $\phi^k$ is achieved: $
F\in\tilde{\cal F}^k\Leftrightarrow \ F\in{\cal F}^k$ and $\Upsilon^F=\phi^k$. Forests from $\tilde{\cal F}^k$ are called minimal. 

${\cal F}^k|_{\cal S}$ --- the set of subgraphs of $k$-component forests induced by the set ${\cal S}$;

$\tilde{\cal F}^k|_{\cal S}$ --- the set of subgraphs of $k$-component minimum weight forests induced by the set ${\cal S}$; 

$F^G_{\uparrow{\cal S}}$ --- the graph obtained from $F$ by replacing the arcs coming from the vertices of the set ${\cal S}$ with arcs coming from the same vertices in  graph $G$.

For any subset ${\cal S}\subset {\cal N}$, its complement $\overline{\cal S}={\cal N}\setminus {\cal S}$.

Forest $F\in{\cal F}^{k}$ is called a {\it pseudo-ancestor} of a forest $G\in{\cal F}^{k-1}$, and $G$ is a {\it pseudo-descendant} of $F$, if there is a root $j$ of $F$ such that: ${\cal K}_G={\cal K}_F\setminus \{ j \}$ and $G|_{{\cal V}T^F_q}=T^F_q$ for $q\in{\cal K}_G$. Such forests are called pseudo-related, and the set of such pseudo-descendants of $F$ is denoted by ${\cal P}^F_j$. 

Forest $F\in{\cal F}^{k}$ is called an {\it ancestor} of a forest $G\in{\cal F}^{k-1}$, and $G$ is called a {\it descendant} of a forest $F$, if for some $j\in{\cal K}_F$ the following holds: $G\in{\cal P}^F_j$ and $G|_{{\cal V}T^F_j}$ is a tree. Forests $F$ and $G$ are called related, and the set of such descendants of a forest $F$ is denoted by ${\cal R}^F_j$. The set of all descendants of a forest $F$ is denoted by ${\cal R}^F=\underset{j\in {\cal K}_F}{\cup}{\cal R}^F_j$. 

\section{Properties used}
\subsection{Arc replacement operation} 

We will need a simple property \cite[Corollary 2 of Lemma 1]{V6} of the operation of replacing arcs ing from the vertices of a certain set in two arbitrary forests. Let us formulate it in a somewhat broader setting. 

\begin{property} \cite{V6} {\it Let $F$ and $G$ be forests, ${\cal V}G \subseteq {\cal V}F$, let ${\cal D}\subseteq {\cal V}G$ be a subset of the vertex set that is not entered by arcs in the forest $F$. Then graph $F_{\uparrow\cal D}^G$ is a forest.} 
\end{property}
A thinner one will also be used.
\begin{property} \cite[Corollary 6 from Lemma 1]{V6} {\it Let $F$ and $G$ be forests with the same vertex set, and let $T^F$ and $T^G$ be trees of the forests of $F$ and $G$, respectively, and let ${\cal D}\subset{\cal V}T^F\cap{\cal V}T^G$, such that ${\cal N}^{in}_{\cal D}(G)=\emptyset$ and ${\cal N}^{out}_{\cal D}(G)\subset{\cal V}T^G\setminus{\cal V}T^F $. Then the graphs $F_{\uparrow\cal D}^G$ and $G_{\uparrow\cal D}^F$ are forests.}
\end{property}

\subsection{Minima of weight on subsets of a vertex set}

In \cite{V8}, for any subset ${\cal S}$ of the set of all vertices ${\cal N}$,  special tree-like minima are defined: 

\begin{equation}
 \lambda_{\cal S}^{\bullet q}=\min_{T\in {\cal T}^{\bullet q}_{\cal S}}\Upsilon^T, \  \lambda_{\cal S}^\bullet= \min_{q\in{\cal S}} \lambda_{\cal S}^{\bullet q},  
\label{bt}
\end{equation} 
where ${\cal T}^{\bullet q}_{\cal S}$ is a set of trees with vertex set ${\cal S}$ and root at vertex $q\in {\cal S}$. The second of the minima (\ref{bt}) can be rewritten as
\begin{equation}
\lambda_{\cal S}^\bullet=\min_{T\in {\cal T}^{\bullet}_{\cal S}}\Upsilon^T \ , \ \ {\cal T}^\bullet_{\cal S}= 
\underset{q\in{\cal S}}{\cup}{\cal T}^{\bullet q}_{\cal S} \ , 
\label{Taus}
\end{equation}
using the disjoint union of sets ${\cal T}^{\bullet q}_{\cal S}$ by vertices $q\in{\cal S}$. 

Let us also identify subsets of trees on which the corresponding minima are achieved: $T\in\tilde{\cal T}^{\bullet q}_{\cal S}$ if $T\in{\cal T}^{\bullet q}_{\cal S}$ and $\Upsilon^T=\lambda_{\cal S}^{\bullet q}$; $T\in\tilde{\cal T}^{\bullet }_{\cal S}$ if $T\in{\cal T}^{\bullet }_{\cal S}$ and $\Upsilon^T=\lambda_{\cal S}^{\bullet }$. 

Also in \cite{V8} forest-like minima are defined
   
\begin{equation}
 \mu_{\cal S}^{\bullet q}=\min_{F\in {\cal F}^{\bullet q}_{\cal S}}\Upsilon^F_{\cal S}\ , \ \ \mu_{\cal S}^\bullet= \min_{q\in{\cal S}} \mu_{\cal S}^{\bullet q}, 
\label{bf}
\end{equation} 
where ${\cal F}^{\bullet q}_{\cal S}$ is the set of spanning forests such that if $F\in {\cal F}^{\bullet q}_{\cal S}$, then the set ${\cal S}$ contains exactly one root of forest $F$, namely vertex $q\in{\cal S}$. In this case, the entries of arcs coming in the forest $F$ from the vertices of the set ${\cal S}$ may or may not belong to the set ${\cal S}$ itself. Arcs coming in $F$ from the vertices of $\overline{\cal S}$ do not affect the value $\mu_{\cal S}^{\bullet q}$.  

For these quantities the following is true

\begin{property}\cite[Proposition 2]{V8}
{\it Let $F\in\tilde{\cal F}^k$ and $l\in{\cal K}_F$, 
then }
\begin{equation}
\mu_{\cal S}^{\bullet l} = \mu_{\cal S}^{\bullet }= 
\lambda_{\cal S}^\bullet = \lambda_{\cal S}^{\bullet l} =\Upsilon^F_{\cal S} \ ,  
\label{mlu}
\end{equation}
{\it where}  ${\cal S}={\cal V}T^F_l$.   
\end{property}

Forest-like minima are also defined for the case when the set ${\cal S}$ does not contain forest roots: 

\begin{equation}
\mu^{\circ }_{\cal S} = \min_{F\in{\cal F}^\circ_{{\cal S}}} \Upsilon^F_{\cal S} \ , 
\label{cf} 
\end{equation} 
where ${\cal F}^{\circ}_{\cal S}$ is a set of spanning forests in which arcs come from all vertices of the set ${\cal S}$. Arcs come from the vertices of $\overline{\cal S}$ in any forest $F\in {\cal F}^\circ_{\cal S}$ are arbitrary and do not affect the weight value. 

For all introduced weights, both tree-like (type $\lambda$) and forest-like (type $\mu$), we assume that they are equal to infinity if the corresponding set of trees (${\cal T}$) and, accordingly, forests (${\cal F}$) is empty.  

\section{Tree-like weights on a subset not containing roots}

We are further interested in the situation when arcs come from all vertices of the set ${\cal S}$ and they form a special tree. Namely, we introduce a set of trees ${\cal T}^{\circ q}_{\cal S}$, which are subgraphs of the original graph $V$, and if $T\in{\cal T}^{\circ q}_{\cal S}$, 
then $|{\cal V}T|=|{\cal S}|+1$ and  $T|_{\cal S}\in{\cal T}^{\bullet q}_{\cal S}$. This means that in the tree $T$, from the set ${\cal S}$ itself, there is exactly one arc, and it comes from the vertex $q$. Let this be the arc $(q,r)$. Obviously, $r\in\overline{\cal S}$ and this vertex is the root of the tree $T$. Despite the fact that $r$ is the root, this vertex is not of particular interest in the context of this work. Therefore, it is not an index in the set of trees under consideration. But the vertex $q$ is more significant, since the subject of consideration is the arcs coming from the vertices of the set ${\cal S}$, to which the vertex $q$ belongs, unlike $r$. Note that the weights of the trees of the corresponding trees $T$ and $T_q=T|_{\cal S}$ are related by the ratio

\begin{equation}
\Upsilon^{T}=\Upsilon^{T_q}+v_{qr}. 
\label{U}
\end{equation}

 Let us also introduce the disjoint union

\begin{equation}
{\cal T}^{\circ }_{\cal S}=\mathop\cup_{q\in{\cal S}}{\cal T}^{\circ q}_{\cal S}\ 
\end{equation}
  and determine the weights

\begin{equation}
\lambda_{\cal S}^\circ =\min_{T\in {\cal T}_{\cal S}^\circ}\Upsilon^{T} , \ \ \lambda_{\cal S}^{\circ q}=\min_{T\in {\cal T}^{\circ q}_{\cal S}}\Upsilon^{T} \ .
\label{c}
\end{equation} 
Obviously
\begin{equation}
\lambda_{\cal S}^\circ =\min_{q\in {\cal S}} \lambda_{\cal S}^{\circ q} \ .
\label{cm}
\end{equation}  
Note that for different trees $T$ and $T'$ from ${\cal T}^\circ_{\cal S}$, the sets of their vertices, generally speaking, do not coincide, since the root of the tree can be any vertex from $\overline{\cal S}$.

Let us find out the connections between the introduced weights. 
If the tree $T\in{\cal T}_{\cal S}^\circ$ is supplemented with the remaining vertices to a spanning subgraph, then we obtain a spanning forest belonging to the set ${\cal F}_{\cal S}^\circ$. Therefore, the following inequality is satisfied
  
\begin{equation}
\mu^\circ_{\cal S} \le \lambda^\circ_{\cal S}  \ .
\label{mll}
\end{equation} 
Note that for an arbitrary ${\cal S}$, \cite{V8} a similar inequality $
\mu^\bullet_{\cal S} \le \lambda^\bullet_{\cal S}$ holds. 
Equality (\ref{mlu}) holds if ${\cal S}$ is the vertex set of some tree in the minimal forest. A similar situation holds for inequality (\ref{mll}).

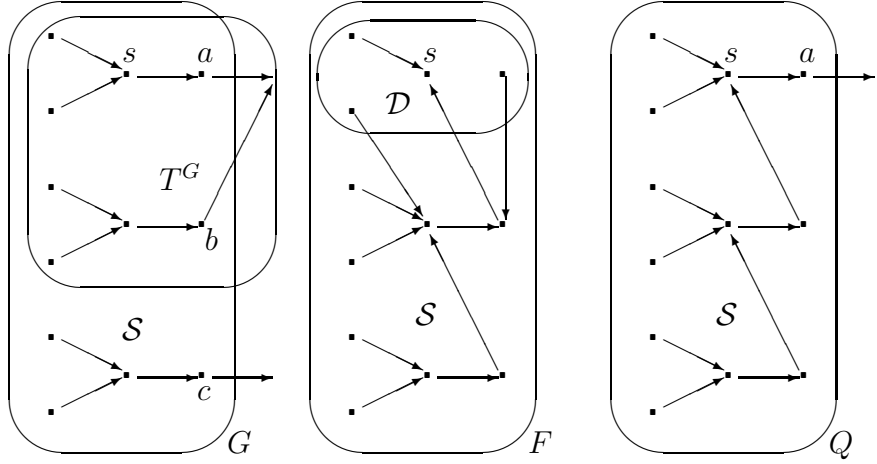
\begin{figure}[h]
\unitlength=1mm
\begin{center}
\begin{picture}(115,60)

\put(25,52){$a$}
\put(26,27){$b$}
\put(25,7){$c$}

\put(15,52){$s$}
\put(20,35){$T^G$}
\put(15,15){${\cal S}$}
\put(29,0){$G$}
\put(5,55){$\centerdot$}
\put(5,45){$\centerdot$}
\put(5,35){$\centerdot$}
\put(5,25){$\centerdot$}
\put(5,15){$\centerdot$}
\put(5,5){$\centerdot$}

\put(15,10){$\centerdot$}
\put(15,30){$\centerdot$}
\put(15,50){$\centerdot$}

\put(25,10){$\centerdot$}
\put(25,30){$\centerdot$}
\put(25,50){$\centerdot$}

\put(7,6){\vector(2,1){8}}
\put(7,26){\vector(2,1){8}}
\put(7,46){\vector(2,1){8}}

\put(7,15){\vector(2,-1){8}}
\put(7,35){\vector(2,-1){8}}
\put(7,55){\vector(2,-1){8}}

\put(17,10){\vector(1,0){8}}
\put(17,30){\vector(1,0){8}}
\put(17,50){\vector(1,0){8}}

\put(27,10){\vector(1,0){8}}
\put(26,31){\vector(1,2){9}}
\put(27,50){\vector(1,0){8}}

\put(19,40){\oval(33,36)}
\put(15,30){\oval(30,60)}

\put(55,52){$s$}
\put(69,0){$F$}
\put(50,45){${\cal D}$}
\put(54,17){${\cal S}$}
\put(45,55){$\centerdot$}
\put(45,45){$\centerdot$}
\put(45,35){$\centerdot$}
\put(45,25){$\centerdot$}
\put(45,15){$\centerdot$}
\put(45,5){$\centerdot$}

\put(55,10){$\centerdot$}
\put(55,30){$\centerdot$}
\put(55,50){$\centerdot$}

\put(65,10){$\centerdot$}
\put(65,30){$\centerdot$}
\put(65,50){$\centerdot$}

\put(47,6){\vector(2,1){8}}
\put(47,26){\vector(2,1){8}}
\put(46,45){\vector(2,-3){9}}

\put(47,15){\vector(2,-1){8}}
\put(47,35){\vector(2,-1){8}}
\put(47,55){\vector(2,-1){8}}

\put(57,10){\vector(1,0){8}}
\put(57,30){\vector(1,0){8}}

\put(65,11){\vector(-1,2){9}}
\put(65,31){\vector(-1,2){9}}
\put(66,50){\vector(0,-1){19}}
\put(55,50){\oval(28,15)}
\put(55,30){\oval(30,60)}

\put(105,52){$a$}
\put(95,52){$s$}
\put(109,0){$Q$}
\put(94,17){${\cal S}$}
\put(85,55){$\centerdot$}
\put(85,45){$\centerdot$}
\put(85,35){$\centerdot$}
\put(85,25){$\centerdot$}
\put(85,15){$\centerdot$}
\put(85,5){$\centerdot$}

\put(95,10){$\centerdot$}
\put(95,30){$\centerdot$}
\put(95,50){$\centerdot$}

\put(105,10){$\centerdot$}
\put(105,30){$\centerdot$}
\put(105,50){$\centerdot$}

\put(87,6){\vector(2,1){8}}
\put(87,26){\vector(2,1){8}}
\put(87,46){\vector(2,1){8}}

\put(87,15){\vector(2,-1){8}}
\put(87,35){\vector(2,-1){8}}
\put(87,55){\vector(2,-1){8}}

\put(97,10){\vector(1,0){8}}
\put(97,30){\vector(1,0){8}}
\put(97,50){\vector(1,0){8}}

\put(105,11){\vector(-1,2){9}}
\put(105,31){\vector(-1,2){9}}
\put(107,50){\vector(1,0){8}}

\put(95,30){\oval(30,60)}

\end{picture} 
\caption{\small On the left are the arcs of the forest $G$ coming from the vertices of the set ${\cal S}={\cal V}T^F_s$; in the center are the arcs of the tree $T^F_s$ and the set ${\cal D}$ of vertices of the connected component of the induced subgraph $G|_{\cal S}$, including the vertex $s$; on the right are the arcs of the graph $Q=F^G_{\uparrow{\cal D}}$ coming from the vertices of the set ${\cal S}$. } 
\label{r1}
\end{center}
\end{figure}

\begin{theorem}
{\it Let $F\in\tilde{\cal F}^k$ and $s\in{\cal K}_F$, 
then }
\begin{equation}
\mu_{\cal S}^\circ = 
\lambda_{\cal S}^\circ ,  
\label{mel}
\end{equation}
{\it where}  ${\cal S}={\cal V}T^F_s$.
\end{theorem}

\begin{proof}
First of all, we note that if $\mu^\circ_{\cal S}=\infty$, then by (\ref{mll}) $\lambda_{\cal S}^\circ=\infty$. Therefore, we consider only the situation when $\mu^\circ_{\cal S}<\infty$. Let the spanning forest $G\in{\cal F}^\circ_{\cal S}$ be such that $\Upsilon^{G}_{\cal S}=\mu^\circ_{\cal S}$. For convenience and without loss of generality, we assume that no arcs come from the vertices of the set $\overline{\cal S}$ ($\Upsilon^{G}=\Upsilon^{G}_{\cal S}$). 

The induced subgraph $G|_{\cal S}$ is a forest on the vertex set ${\cal S}$. The roots of its trees are those vertices of the set ${\cal S}$ from which arcs in the graph $G$ itself originate and enter $\overline{\cal S}$ (in Fig. \ref{r1} these are vertices $a$, $b$, and $c$). Let ${\cal D}$ be the set of vertices of the tree of the induced forest $G|_{\cal S}$ that contains the root $s$ of the forest $F$, and let $T^G$ be the tree of the forest $G$ that contains ${\cal D}$.

Let us introduce graphs $Q=F^G_{\uparrow{\cal D}}$ and $R=G^F_{\uparrow{\cal D}}$. The set ${\cal D}$ satisfies the conditions of Property 2, so both graphs $Q$ and $R$ are forests. Note that $Q\in{\cal R}^F_s$. 

In the forest $Q$, arcs come from all vertices of the set ${\cal S}$, and exactly one arc comes from the set ${\cal S}$ itself (in Fig.\ref{r1} on the right, it comes from vertex $a$). Thus, $Q|_{\cal S}$ is a tree with a root at $a$. If we supplement it with an arc coming from $a$ in the forest $Q$, we obtain a tree belonging to the set ${\cal T}^\circ_{\cal S}$. Thus $\Upsilon_{\cal S}^Q\geq \lambda^\circ_{\cal S}$. We have

\begin{equation}
\lambda^\circ_{\cal S}\leq \Upsilon^Q_{\cal S}=\Upsilon^F_{\cal S}-\Upsilon^F_{\cal D}+\Upsilon^G_{\cal D}=
\lambda^\bullet_{\cal S} -\Upsilon^F_{\cal D}+\Upsilon^G_{\cal D} \ .
\label{llu}
\end{equation}
In the forest $R$ arcs originate from all the vertices of the set ${\cal S}$ except one (this is the vertex $s$), so $\Upsilon_{\cal S}^R\ge\mu^\bullet_{\cal S}$ is executed.
On the other hand, considering (\ref{mlu})
\begin{equation}
\lambda^\bullet_{\cal S}=\mu^\bullet_{\cal S}\leq \Upsilon^R_{\cal S}=\Upsilon^G_{\cal S}-\Upsilon^G_{\cal D}+\Upsilon^F_{\cal D}=
\mu^\circ_{\cal S} -\Upsilon^G_{\cal D}+\Upsilon^F_{\cal D} \ .
\label{lmu}
\end{equation}
Substituting $\lambda^\bullet_{\cal S}$ from (\ref{lmu}) into (\ref{llu}), we obtain $\lambda^\circ_{\cal S}\leq \mu^\circ_{\cal S} $, and by (\ref{mll}) the reverse inequality holds.
\end{proof}

\begin{proposition} Let ${\cal S}\subsetneq {\cal N}$ and ${\cal T}^\circ_{\cal S}\neq\emptyset$, then   
\begin{equation}
\lambda_{\cal S}^\circ =\min_{q\in {\cal S} }\left(\lambda_{\cal S}^{\bullet q} + \min_{r\notin{\cal S}}v_{qr}\right) \ . 
\label{lo}
\end{equation}
 \end{proposition}

\begin{proof}
Let $q\in{\cal S}$. By definition, if $T\in {\cal T}_{\cal S}^{\circ q}$, then its subgraph $T|_{\cal S}$ induced by the set $\cal S$ is a tree. 
Let $T_q=T|_{\cal S}$. Then
\begin{equation}
\Upsilon^{T}=\Upsilon^{T_q}+v_{qr}, 
\label{Ut}
\end{equation}
where $(q,r)\in {\cal A}T$. Since $T_q$ is a tree, this arc is the only one coming from $q$ at $T$. Moreover, $r\notin{\cal S}$ is the root of the tree $T$, and $T_q\in{\cal T}^{\bullet q}_{\cal S}$. Let us determine the minimum of (\ref{Ut}) for a fixed $q$.

\begin{equation}
\lambda_{\cal S}^{\circ q} =  \min_{T\in{\cal T}_{\cal S}^{\circ q}}\Upsilon^{T} =\min_{\substack{ r\notin {\cal S} \\ T_q \in{\cal T}^{\bullet q}_{\cal S} }} \left(\Upsilon^{T_q} + v_{qr}\right) \ .
\end{equation} 

Since $r\notin {\cal S}$ and this vertex is not related to the weight of the tree $T_q$, the expression for the minimum takes the form
\begin{equation*}
\lambda_{\cal S}^{\circ q} =
 \min_{T_q\in{\cal T}^{\bullet q}_{\cal S}} \Upsilon^{T_q} + \min_{r\notin{\cal S}}v_{qr}  \ .
\end{equation*}
Given the definition (\ref{bt}), we obtain
 
\begin{equation}
\lambda_{\cal S}^{\circ q}=\lambda_{\cal S}^{\bullet q} + \min_{r\notin{\cal S}}v_{qr} \ . 
\label{loq}
\end{equation}
 In turn, (\ref{cm}) now implies (\ref{lo}). 
\end{proof}

\section{Relatedness of forests}

The main theorem \cite[Theorem 2 (on related forests)]{V6} we formulate as a property

\begin{property} {\it Let ${\cal F}^{k-1}\neq \emptyset$, then any forest from $\tilde{\cal F}^{k}$ has a descendant in $\tilde{\cal F}^{k-1}$ and vice versa --- any forest from $\tilde{\cal F}^{k-1}$ has an ancestor in $\tilde{\cal F}^{k}$.} 
\end{property}

In \cite[Theorem 2]{V8} a criterion is formulated for a pseudo-descendant of a minimal forest to also be minimal, and on this basis an efficient algorithm for constructing minimal pseudo-sibling forests is presented. We present the formulation in the form of a property.

\begin{property}
Let $F\in \tilde{\cal F}^{k}$, $j\in{\cal K}_F$ and $G\in{\cal P}^F_j$. Then in order for $G\in \tilde{\cal F}^{k-1}$ it is necessary and sufficient that  
\begin{equation}
\Upsilon^G_{{\cal V}T^F_j}=\mu^\circ_{{\cal V}T^F_j}
\label{ul}
\end{equation}
 and
\begin{equation}
\mu^\circ_{{\cal V}T^F_j}- 
\mu^\bullet_{{\cal V}T^F_j} = \min_{l\in{\cal K}_F} \left( \mu^\circ_{{\cal V}T^F_l}- 
\mu^\bullet_{{\cal V}T^F_l}  \right)  \ . 
\label{potom}
\end{equation}
\end{property}

Let us prove a similar criterion for related forests.

\begin{theorem}
Let $F\in \tilde{\cal F}^{k}$, $y\in{\cal K}_F$ and $G\in{\cal R}^F_y$. In order for $G\in \tilde{\cal F}^{k-1}$, it is necessary and sufficient that  
\begin{equation}
\Upsilon^G_{{\cal V}T^F_y}=\lambda^\circ_{{\cal V}T^F_y}
\label{ull}
\end{equation}
 and
\begin{equation}
\lambda^\circ_{{\cal V}T^F_y}- 
\lambda^\bullet_{{\cal V}T^F_y} = \min_{l\in{\cal K}_F} \left( \lambda^\circ_{{\cal V}T^F_l}- 
\lambda^\bullet_{{\cal V}T^F_l}  \right)  \ . 
\label{potoml}
\end{equation}
\end{theorem}

\begin{proof}
 By Property 3 and Theorem 1, for all trees in forest $F$,
\begin{equation}
\lambda^\bullet_{{\cal V}T^F_l}=\mu^\bullet_{{\cal V}T^F_l} \ , \ \ \   
\lambda^\circ_{{\cal V}T^F_l}=\mu^\circ_{{\cal V}T^F_l} \ , \ \ \ l\in{\cal K}_F \ .
\end{equation}
Forest $G$ is a descendant of forest $F$, and in particular is a pseudo-descendant. If forest $G$ and root $y$ satisfy conditions (\ref{ull}) and (\ref{potoml}), then forest $G$ also satisfies conditions (\ref{ul}) and (\ref{potom}). Thus, by Property 5, forest $G$ is a minimal pseudo-descendant of forest $F$. The weights of all minimal forests with a fixed number of roots coincide. So the forest $G$ is a minimal descendant. In the opposite direction. Let $G$ be a minimal descendant and $G\in{\cal R}^F_j$. Then it is a minimal pseudo-descendant. By Property 4, conditions (\ref{ul}) and (\ref{potom}) are satisfied for it. That is, conditions (\ref{ull}) and (\ref{potoml}).  
\end{proof}
 
\begin{corollary}
Let $F\in\tilde{\cal F}^{k}$, then
\begin{equation}
\phi^{k-1}-\phi^{k}=\min_{l\in{\cal K}_F} \left( \lambda^\circ_{{\cal V}T^F_l}- 
\lambda^\bullet_{{\cal V}T^F_l}  \right) \ . 
\label{vyp}
\end{equation}
\end{corollary}

\begin{proof}
Let $G\in \tilde{\cal F}^{k-1}$ be a descendant of forest $F$ and $\{ y\} ={\cal K}_F\setminus{\cal K}_G$. That is, $y$ is the root on which the minimum on the right-hand side of (\ref{vyp}) is reached, that is, (\ref{potoml}) is satisfied. We have: $\phi^{k}=\Upsilon^F$, $\phi^{k-1}=\Upsilon^G$. Forests $F$ and $G$ differ only in the arcs coming from the vertices of the set ${\cal V}T^F_y$. Thus, taking into account (\ref{ull}) 

\begin{equation}
\phi^{k-1}-\phi^{k}=\Upsilon^G-\Upsilon^F= \Upsilon^G_{{\cal V}T^F_y}-\Upsilon^F_{{\cal V}T^F_y}= \lambda_{{\cal V}T^F_y}^\circ-\lambda_{{\cal V}T^F_y}^\bullet  \ . 
\end{equation}
Taking into account (\ref{potoml}), we obtain (\ref{vyp}). 
\end{proof}

\section{Recalculation of minimum weights when enlarging trees} 

\begin{proposition}
Let $F\in\tilde{\cal F}^{k}$, $ y \in{\cal K}_F$, $G\in\tilde{\cal F}^{k-1}\cap {\cal R}_y^F$. Let also the entry of the only arc coming from the set ${\cal V}T_y^F$ at $G$ belong to the set ${\cal V}T^F_x$. Then

\begin{equation}
\lambda_{{\cal V}T^G_x}^\bullet=\lambda_{{\cal V}T^F_x}^\bullet+\lambda_{{\cal V}T^F_y}^\circ \ .
\label{pere}
\end{equation}
  \end{proposition}

\begin{proof}
First of all, we note that in the forest $G$ all trees except the tree rooted at $x$ are the same as in the forest $F$ and nothing has changed for the weight minima on their vertex sets. Since the forest $G$ itself is minimal, then by Property 3 for the vertex set ${\cal S}$ of any tree in it, $\lambda_{\cal S}^\bullet =\Upsilon^G_{\cal S}$ holds.
The only tree in $G$ with a vertex set different from that in $F$ is the tree $T^G_x$. Since $G\in{\cal R}^F_y$, then ${\cal V}T^G_x={\cal V}T^F_x\cup{\cal V}T^F_y$ and $G|_{{\cal V}T^F_x}={\cal V}T^F_x$. Taking into account (\ref{ul}), we have:

\begin{equation*}
\lambda_{{\cal V}T^G_x}^\bullet =\Upsilon^G_{{\cal V}T^G_x}=\Upsilon^G_{{\cal V}T^F_x}+\Upsilon^G_{{\cal V}T^F_y}=\Upsilon^F_{{\cal V}T^F_x}+\Upsilon^G_{{\cal V}T^F_y}=\lambda_{{\cal V}T^F_x}^\bullet+\lambda_{{\cal V}T^F_y}^\circ \ .
\end{equation*}
\end{proof}

Propositions 1 and 2 together with Theorem 2 allow us to algorithmically construct related minimal forests. 
However, direct use of (\ref{lo}) in finding the values of   $\lambda^\circ_{\cal S}$ (and the corresponding trees) is not effective, since it is required to sequentially assign each vertex $q$ of the set ${\cal S}={\cal V}T^F_y$ as a root.
Property 5 allows us to avoid this procedure and define more convenient for calculation minima $\mu^\circ_{\cal S}$ and the corresponding forests \cite{V8}.
However, there is an important, especially in the physics of Brownian motion in a potential drift field (\cite{V1}, \cite{V2}), special case of weight function, when it is more efficient to calculate the values  $\lambda^\circ_{\cal S}$ and find the trees corresponding to them. 

\section{Digraph of (potential) barriers and graph of potential} 

\subsection{Clarification of the use of terminology}

Further, both directed and undirected graphs will be used, which means that each time it is necessary to stipulate whether an directed graph or a graph (undirected) is used. Therefore, now the widespread use of the term graph (as well as the terms forest and tree) becomes inconvenient and can lead to confusion. The term graph itself now implies an undirected graph.  
Everything that was previously called forests and trees, in the strict sense, are directed forests and directed trees, and entering ones at that. Unless specifically noted or clear from the context, the terms forest and tree now refer to undirected graphs. Namely, a forest is an acyclic graph, and a connected component of a forest is a tree.

\subsection{Notations and definitions for undirected graphs}

Undirected graphs will be denoted by a sans-serif font. Along with the original weighted directed graph $V$, there will be an associated undirected graph {\sf P} with loops on the same set of vertices: ${\cal V}{\sf P}={\cal N}$.
The set of edges (and loops) is denoted by ${\cal E}{\sf P}$. Depending on the context, the pair $(i,j)$, for $i\neq j$, can be considered both as an arc of the directed graph $V$ and as an edge of the graph {\sf P}. For the graph {\sf P}, the pairs $(i,j)$ and $(j,i)$ define the same edge with weight $p_{ij}$.

For a subgraph {\sf G} of a graph {\sf P} and a subset of the vertex set ${\cal S}\subseteq {\cal N}$, we introduce weights
\begin{equation}
\Upsilon^{\sf G}_{\cal S}=\sum_{\begin{smallmatrix}i,j\in{\cal S} \\ (i,j)\in {\cal E}{\sf G}\end{smallmatrix}} p_{ij} \ , \ \ \Upsilon^{\sf G}=\Upsilon^{\sf G}_{\cal N}=\sum_{(i,j)\in {\cal E}{\sf G}} p_{ij} \ . \label{vesp}
\end{equation}
Here, in contrast to the oriented situation (\ref{ves}), always $\Upsilon^{\sf G}_{\cal S}=\Upsilon^{{\sf G}|_{\cal S}} $, but there is no additivity property.

The sets of spanning forests of the graph {\sf P} consisting of $k$ trees are denoted by ${\textsc F}^k$. Forests {\sf F} from ${\textsc F}^k$ on which the minimum of weight $\Upsilon^{\sf F}$ is achieved are called minimal. The set of such forests is denoted by $\tilde{\textsc F}^k$.

Similarly, if ${\cal S}\subseteq {\cal N}$, then by ${\textsc T}_{\cal S}$ we mean the set of trees {\sf T} that are subgraphs of {\sf P} such that ${\cal V}{\sf T}={\cal S}$. Trees ${\sf T}\in {\textsc T}_{\cal S}$ on which the minimum of weight $\Upsilon^{\sf T}$ is achieved are called minimal on the set ${\cal S}$. The set of such trees is denoted by $\tilde{\textsc T}_{\cal S}$, and the minimum of weight itself is denoted by $\nu_{\cal S}$:

\begin{equation}
\nu_{\cal S}=\min_{{\sf T}\in{\textsc
T}_{\cal S}}\Upsilon^{\sf T} \ .  
\label{nu}
\end{equation}

\subsection{Relationship between the barrier graph $V$ and the potential graph {\sf P}}

\begin{figure}[h]
\unitlength=1mm
\begin{center}
\begin{picture}(120,85)
\qbezier(0,82)(7,84)(15,45)
\qbezier(15,45)(28,-20)(41,45)
\qbezier(41,45)(50,88)(63,47)
\qbezier(63,47)(71, 23)(79,45)
\qbezier(79,45)(87, 63)(93,45)
\qbezier(93,45)(105, 5)(110,45)
\qbezier(110,45)(115, 84)(120,82)

\put(0,3){\vector(0,1){85}}
\put(0,3){\vector(1,0){120}}
\put(2,85){$P(x)$}

\put(27,0){$x_i$}
\put(26,39){$v_{ij}$}
\put(28,42){\vector(0,1){25}}
\put(28,37){\vector(0,-1){24}}

\put(49,0){$x_{ij}$}
\put(69,0){$x_j$}
\put(85,0){$x_{jk}$}
\put(102,0){$x_k$}

\put(0,67){\line(1,0){75}}
\put(0,12){\line(1,0){30}}
\put(0,34){\line(1,0){75}}

\put(70,50){$v_{ji}$}
\put(71,53){\vector(0,1){14}}
\put(71,49){\vector(0,-1){14}}

\put(1,14){$p_{ii}$}
\put(1,36){$p_{jj}$}
\put(1,69){$p_{ij}$}

\end{picture}
\caption{\small Construction of an undirected graph {\sf P} and an directed graph $V$ based on the potential $P(x)$. }
\label{pot}
\end{center}
\end{figure}
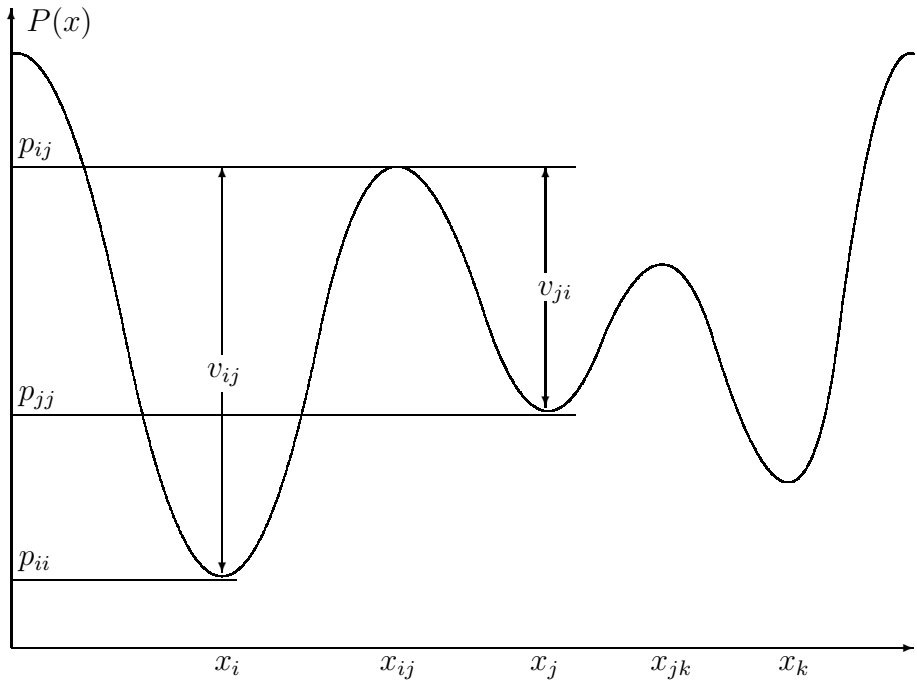

Let {\sf P} be a reflexive (each vertex is incident to a loop) connected weighted undirected graph. We will call such a graph a {\it potential graph}.

On the same vertex set, each potential graph can be assigned a uniquely weighted directed graph $V$ without loops according to the rule: arc $(i,j)\in{\cal A}V$, $i\neq j$, if there is an edge $(i,j)\in {\cal E}{\sf P}$. The weight of each arc is equal to

\begin{equation}
v_{ij}=p_{ij}-p_{ii} \ . 
\label{poten}
\end{equation} 

We will call such an directed graph $V$ {\it an directed graph of (potential) barriers}. It follows from the definition that the directed graph of barriers is not simply strongly connected, but, moreover, if $(i,j)\in{\cal A}V$, then $(j,i)\in{\cal A}V$.

Given a barrier digraph $V$, the relation (\ref{poten}) does not uniquely determine the potential graph {\sf P}. If the same number is added to all edge weights (including loops) of the graph {\sf P}, then the corresponding graph $V$ will be the same. This ambiguity is unimportant for further constructions. 

Let us explain the definitions using the example of a physical potential $P(x)$ of sufficient smoothness (see Fig. \ref{pot}). The minimum points $x_i$ of the function $P(x)$ determine the vertices $i$ of the graph {\sf P}. The potential values at these points correspond to the weights of the loops, and the values at the saddle points of the $x_{ij}$  between the regions of attraction of the dynamical system
 $\dot x=-\nabla P(x)$, having a common boundary (in one dimension --- at points of local maxima), determine the weights of the edges. Exactly
 
\begin{equation*}
p_{ij}=P(x_{ij}) \ , \ \ \ p_{ii}=P(x_i)\ . 
\end{equation*}  

 Note also that in the situation of physical potential, the saddle points surrounding vertex $i$ must have a value $P(x)$ higher than the corresponding minimum point. Then all potential barriers are positive ($v_{ij}>0$, for $(i,j)\in{\cal A}V$). For the weights of the graph {\sf P} this means
 
\begin{equation*}
 p_{ii}< p_{ij} \ , \ \  (i,j)\in{\cal E}{\sf P} \ , i\neq j \ .
\end{equation*} 
If the last property is not satisfied, then, although the graph {\sf P} does not correspond to the function $P(x)$ with the corresponding local minima and saddle points (and the arcs of the barrier graph $V$ can have a negative weight), nevertheless this circumstance is not reflected in the properties under study.

Let us note one more property of the weights of the barrier graph $V$. From (\ref{poten}) we have: $v_{ij}-v_{ji}=p_{jj}-p_{ii}$.  Therefore, if an edge is $(i,j)\in{\cal E}{\sf P}$, then $v_{ij}<v_{ji}$ if and only if $p_{ii}>p_{jj}$.    

If {\sf P} is an arbitrary connected weighted graph, then it can be turned into a reflexive graph by simply adding the missing loops and assigning arbitrary weights to them. The resulting graph is already a potential graph and is uniquely associated with the barrier directed graph $V$. However, with respect to the original graph $P$, such a mapping is obviously no longer unambiguous and there is arbitrariness associated with the choice of loop weights. 

\subsection{Potential shift}

\begin{assertion} Let two undirected graphs {\sf P} and ${\sf P}'$ differ only by a shift of a fixed value $d$ in the weight function: $p_{ij}=p'_{ij}+d$. Let also ${\sf F}$ and ${\sf F}'$ be two $k$-component spanning forests of the graphs {\sc P} and ${\textsc P}'$, respectively. And let ${\cal E}{\sf F}={\cal E}{\sf F}'$. Then the forests {\sf F} and ${\sf F}'$ are simultaneously minimal or not.
\end{assertion}

\begin{proof}
Indeed, an $N$-vertex undirected forest consisting of $k$ components contains exactly $N-k$ edges. Therefore, for an arbitrary $k$-component spanning forest of ${\sf P}$ and the corresponding forest ${\sf F}'$ (with the same set of edges), which is a subgraph of ${\sf P}'$, we have

\begin{equation*}
\Upsilon^{\sf F}-\Upsilon^{\sf F'}=d(N-k) \ .
\end{equation*}
Thus, when shifting the weight function, $k$-component forests change their weight by the same fixed value. Therefore, for any two forests {\sf F} and {\sf G} of ${\textsf P}$ and the corresponding forests ${\sf F}'$ and ${\sf G}'$ of ${\sf P}'$
\begin{equation*}
\Upsilon^{\sf F}-\Upsilon^{\sf G}=\Upsilon^{\sf F'}-\Upsilon^{\sf G'} \ .
\end{equation*}
\end{proof}
By virtue of this statement, the shift of the potential does not affect the structure of its minimal forests and trees. To find minimal forests of the barrier graph $V$, it is sufficient to use any chosen potential graph {\sf P}, related to $V$ by the relation (\ref{poten}). Therefore, below we fix the potential graph {\sf P} by an arbitrary representative.

\subsection{Removal and giving orientation of trees with replacement of weights}

First of all, let us note that if there is an undirected tree {\sf T} on the set of vertices ${\cal S}\subseteq{\cal N}$ (${\cal V}{\sf T}={\cal S}$), then by choosing a vertex $q\in{\cal S}$, it can be turned into an entering tree $T_q$ in a unique way, simply by appointing this vertex as the root. The directions of the arcs are determined automatically. The opposite action is even simpler. An directed tree $T_q$ is uniquely transformed into an undirected tree {\sf T} by simply replacing all arcs with edges. However, in our situation there are two distinguished weighted graphs - the barrier directed graph $V$ and the unoriented potential graph {\sf P}. Therefore, it is natural to perform the operations of removing and giving orientation with a simultaneous replacement of weights. 

Let $T_q$ be a subgraph of the barrier digraph $V$ that is an entering tree with a root at the vertex 
$q\in{\cal S}\subseteq{\cal N}$, ${\cal V}T_q={\cal S}$  ($T_q\in{\cal T}_{\cal S}^{\bullet q}$). By {\it removing orientation with weight replacement} we will mean associating it with an undirected tree {\sf T}, in which all arcs $(r,t)$ are transformed into edges with a simultaneous replacement of the weight $v_{rt}$ with $p_{rt}$. Obviously, ${\sf T}\in{\textsc T}_{\cal S}$. 

Now let {\sf T} be a tree of the potential graph {\sf P} with vertex set 
${\cal S}\subseteq{\cal N}$: ${\sf T}\in{\textsc T}_{\cal S}$ and $q\in{\cal S}$.{\it Giving orientation with weight replacement} is the mapping of the entering tree $T_q$ with the assigned root $q$ to the tree {\sf T}. In this case, if the arc $(i,j)$ belongs to the obtained tree, its weight is considered to be the number $v_{ij}=p_{ij}-p_{ii}$. Obviously, $T_q\in{\cal T}_{\cal S}^{\bullet q}$. Let us agree to call the trees {\sf T} and $T_q$ corresponding to each other.

Note that the operation of removing orientation is carried over to forests $F$ of the barrier graph $V$ without changes. If there is an undirected forest {\sf F}, then to give it orientation it is necessary to specify roots on each connectivity component.  
 
Let us establish a connection between the weights of entering trees on the set ${\cal S}\subseteq{\cal N}$ with different roots and the weights of the corresponding undirected trees.

\begin{assertion}
Let $q\in{\cal S}\subseteq{\cal N}$, $T\in{\cal T}^{\bullet q}_{\cal S}$, and {\sf T} be the corresponding undirected tree. Then

\begin{equation}
\Upsilon^{T_q}=\Upsilon^{\sf T}-\sum_{t\in{\cal S}}p_{tt}+p_{qq} \ . 
\label{TqT}
\end{equation}

\end{assertion}  

\begin{proof}
In the entering tree $T_q$, arcs from the set ${\cal S}$ do not outgo and there are no arcs entering to ${\cal S}$. Therefore, $\Upsilon^{T_q}=\Upsilon^{T_q}_{\cal S}=\Upsilon^{T_q|_{\cal S}}$. We have
\begin{equation}
\Upsilon^{T_q}=\sum_{(r,t)\in{\cal A}T_q}v_{rt}=\sum_{(r,t)\in{\cal E}{\sf T}}(p_{rt}-p_{tt}) \ . 
\end{equation}
In the entering tree $T_q$, a single arc originates from each vertex of the set ${\cal S}$ except the root $q$. Therefore, in the last expression, each value of $p_{tt}$ is present exactly once, with the exception of the missing value of $p_{qq}$. Thus

\begin{equation*}
\Upsilon^{T_q}=\sum_{(r,t)\in{\cal E}{\sf T}}(p_{rt}-p_{tt})=\sum_{(r,t)\in{\cal E}{\sf T}}p_{rt}-\sum_{t\in{\cal S}\setminus \{ q\}}p_{tt}  \ . 
\end{equation*}
In the last expression, the first sum corresponds to the weight of the undirected tree {\sf T} and is equal to $\Upsilon^{\sf T}$. Therefore
\begin{equation}
\Upsilon^{T_q}=\Upsilon^{\sf T}-\sum_{t\in{\cal S}\setminus \{ q\}}p_{tt}  \ . 
\label{TqT'} 
\end{equation}
Considering that
\begin{equation*}
\sum_{t\in{\cal S}\setminus \{ q\}}p_{tt}=\sum_{t\in{\cal S}}p_{tt}-p_{qq} \ , 
\end{equation*}
from (\ref{TqT'}) we obtain (\ref{TqT}).
 
\end{proof}

From the proven statement it immediately follows
\begin{theorem}
Let $q\in{\cal S}\subseteq{\cal N}$, $T_q\in{\cal T}^{\bullet q}_{\cal S}$ and {\sf T} be the corresponding undirected tree. Then $T_q\in\tilde{\cal T}^{\bullet q}_{\cal S}$ if and only if ${\sf T}\in\tilde{\textsc T}_{\cal S}$.
\end{theorem}
\begin{proof}

In the expression (\ref{TqT'}), for a fixed $q$, the value of $\sum\limits_{t\in{\cal S}\setminus \{ q\}}p_{tt}$ is the same for all trees in ${\cal T}^{\bullet q}_{\cal S}$. Therefore, the minimum weights of the trees $T_q$ and {\sf T} are reached simultaneously.
\end{proof}
A special case of this theorem is considered in \cite{Bu-Kh}. 

\begin{assertion}
Let ${\sf T}\in{\textsc T}_{\cal S}$, and let $T_q$ and $T_{q'}$ be the corresponding entering trees with roots at $q$ and $q'$, respectively. Then
\begin{equation}
\Upsilon^{T_q}-p_{qq}=\Upsilon^{T_{q'}}-p_{q'q'} \ . 
\label{qq'}
\end{equation}
\end{assertion}  

\begin{proof}
Writing the expression (\ref{TqT}) for the vertices $q$ and $q'$ and subtracting one from the other, we obtain
\begin{equation*}
\Upsilon^{T_q}-\Upsilon^{T_{q'}}=p_{qq}-p_{q'q'} \ . 
\end{equation*}
which coincides with (\ref{qq'}). 
\end{proof}

\subsection{Weights of type $\lambda$ and type $\nu$  }

\begin{assertion}
Let $\{q,q'\}\subseteq{\cal S}\subseteq{\cal N}$, ${\sf T}\in\tilde{\textsc T}_{\cal S}$,  then 

\begin{equation}
\lambda_{\cal S}^{\bullet q}-p_{qq}=\lambda_{\cal S}^{\bullet q'}-p_{q'q'} \ . 
\label{lqq'}
\end{equation}
\end{assertion}
\begin{proof}

We introduce $T_q$ and $T_{q'}$ --- the corresponding trees to the tree {\sf T} with roots at $q$ and $q'$, respectively.
Note that according to Theorem 3  $T_q\in\tilde{\cal T}^{\bullet q}_{\cal S}$  and  $T_{q'}\in\tilde{\cal T}^{\bullet q'}_{\cal S}$. Thus, $\lambda_{\cal S}^{\bullet q}=\Upsilon^{T_q}$ and $\lambda_{\cal S}^{\bullet q'}=\Upsilon^{T_{q'}}$. Now the statement follows from (\ref{qq'}). 
\end{proof}

\begin{assertion}
Let $q\in{\cal S}\subseteq{\cal N}$,  ${\sf T}\in\tilde{\textsc T}_{\cal S}$. Then  

\begin{equation}
\lambda_{\cal S}^{\bullet q}=\nu_{\cal S}-\sum_{t\in{\cal S}}p_{tt} +p_{qq}\ . 
\label{ln}
\end{equation}
Moreover, if $p_{yy}=\min\limits_{t\in{\cal S}}p_{tt}$, then

\begin{equation}
\lambda_{\cal S}^{\bullet}=\lambda_{\cal S}^{\bullet y}=\Upsilon^{T_y} \ , 
\label{lu}
\end{equation}
where $T_y$ is rooted at $y$ the entering tree corresponding to the tree {\sf T}. 
\end{assertion}

\begin{proof}
By condition $\Upsilon^{\sf T}=\nu_{\cal S}$. According to Theorem 3, for each of the entering trees $T_q$ corresponding to the tree {\sf T}, $\lambda_{\cal S}^{\bullet q}=\Upsilon^{T_q}$ holds. Substituting this value into (\ref{TqT}) from Claim 2, we obtain(\ref{ln}). 

In (\ref{ln}) the only term that depends on $q$ is the last term $p_{qq}$. By assumption, it is minimal when $q=y$. Thus 
\begin{equation*}
\lambda_{\cal S}^{\bullet}=
\min_{q\in{\cal S}}\lambda_{\cal S}^{\bullet q}= 
\min_{q\in{\cal S}}( \nu_{\cal S}-\sum_{t\in{\cal S}}p_{tt}+p_{qq})=
\end{equation*}
\begin{equation*}
=\nu_{\cal S}-\sum_{t\in{\cal S}}p_{tt}+\min_{q\in{\cal S}}p_{qq}=\Upsilon^{\sf T}-\sum_{t\in{\cal S}}p_{tt}+p_{yy}=\Upsilon^{T_y} \ .  
\end{equation*}
The last equality takes into account (\ref{TqT}).
\end{proof}

\begin{assertion}
Let ${\cal S}\subsetneq{\cal N}$,  then 

\begin{equation}
\lambda^\circ_{\cal S}=\nu_{\cal S}-\sum_{t\in{\cal S}}p_{tt}+ 
\min_{\begin{smallmatrix} q\in{\cal S}, \\ r\notin{\cal S} \end{smallmatrix}}p_{qr} \ . 
\label{lnn}
\end{equation}

\end{assertion}

\begin{proof} 

Let's write (\ref{lo}), taking into account (\ref{ln})

\begin{equation*}
\lambda_{\cal S}^\circ =\min_{q\in {\cal S} }\left(\lambda_{\cal S}^{\bullet q} + \min_{r\notin{\cal S}}v_{qr}\right)= \min_{q\in{\cal S}}
\left(\nu_{\cal S}+p_{qq}-\sum_{t\in{\cal S}}p_{tt} +\min_{r\notin{\cal S}}v_{qr}\right) = 
\end{equation*}
\begin{equation*}
=\nu_{\cal S}-\sum_{t\in{\cal S}}p_{tt}+\min_{q\in{\cal S}}\left(p_{qq}+\min_{r\notin{\cal S}}v_{qr}\right)= 
\end{equation*}
\begin{equation*}
=\nu_{\cal S}-\sum_{t\in{\cal S}}p_{tt}+\min_{q\in{\cal S}}\left(p_{qq}+\min_{r\notin{\cal S}}(p_{qr}-p_{qq})\right)=\nu_{\cal S}-\sum_{t\in{\cal S}}p_{tt}+\min_{\begin{smallmatrix} q\in{\cal S}, \\ r\notin{\cal S} \end{smallmatrix}}p_{qr}. 
\end{equation*}

\end{proof}

\begin{assertion}
Let ${\cal S}\subsetneq{\cal N}$, and $p_{ab}=\min\limits_{\begin{smallmatrix} q\in{\cal S}, \\ r\notin{\cal S} \end{smallmatrix}}p_{qr}$, then 

\begin{equation}
\lambda_{\cal S}^\circ=\lambda_{\cal S}^{\bullet a}+v_{ab} \ . \label{lv}
\end{equation}
\end{assertion}

\begin{proof}
Let's add and subtract the value $p_{aa}$ into (\ref{lnn}) and take into account (\ref{ln}): 
\begin{equation*}
\lambda_{\cal S}^\circ=\nu_{\cal S}-\sum_{t\in{\cal S}}p_{tt}+ p_{ab} +p_{aa}-p_{aa}=\lambda_{\cal S}^{\bullet a}+p_{ab}-p_{aa}=\lambda_{\cal S}^{\bullet a}+v_{ab} \ . 
\end{equation*}
\end{proof}

\begin{assertion}
Let $t\in {\cal S}\subsetneq{\cal N}$, $p_{ab}=\min\limits_{\begin{smallmatrix} q\in{\cal S}, \\ r\notin{\cal S} \end{smallmatrix}}p_{qr}$, then 
\begin{equation}
\lambda_{\cal S}^{\circ } =\lambda_{\cal S}^{\bullet t}-p_{tt}+p_{ab} \ . 
\label{lt}
\end{equation}
\end{assertion}

\begin{proof}
Let's write it down (\ref{lv}) and take it into account (\ref{lqq'}): 
\begin{equation*}
\lambda_{\cal S}^\circ=\lambda_{\cal S}^{\bullet a}+v_{ab}=  
\lambda_{\cal S}^{\bullet a}+p_{ab}-p_{aa}= \lambda_{\cal S}^{\bullet a}-p_{aa}+p_{ab}= 
\lambda_{\cal S}^{\bullet t}-p_{tt}+p_{ab}
\ ,  
\end{equation*}
which coincides with (\ref{lt}).

\end{proof}

\begin{assertion} 
Equality is fair

\begin{equation}
\lambda_{\cal S}^\circ -\lambda_{\cal S}^{\bullet}= \min\limits_{\begin{smallmatrix} q\in{\cal S}, \\ r\notin{\cal S} \end{smallmatrix}}p_{qr} -\min\limits_{t\in{\cal S}}p_{tt} \ . 
\label{ll}
\end{equation}
\end{assertion}
 \begin{proof}
 
Let $p_{yy}=\min\limits_{t\in{\cal S}}p_{tt}$.  According to Claim 5 $\lambda_{\cal S}^{\bullet}=\lambda_{\cal S}^{\bullet y}$. The expression (\ref{lt}) takes the form 
\begin{equation}
\lambda_{\cal S}^{\circ } =\lambda_{\cal S}^{\bullet}-p_{yy}+p_{ab} \ . 
\label{lab}
\end{equation}
Substituting the definitions of $p_{yy}$ and $p_{ab}$ into the quantities, we obtain (\ref{ll}).
 \end{proof}

The increment of weights (\ref{ll}) has a natural physical interpretation. It is equal to the value of the potential barrier that must be overcome (or, equivalently, the work that must be done) in order to go beyond the deepest point of the set ${\cal S}$.

\section{On the algorithmization of construction of minimal forests and trees of the barrier graph}

For arbitrary digraphs, there is a unique algorithm for constructing minimum outgoing forests (up to a spanning minimum tree) \cite{V8}. It uses as a subalgorithm the very sophisticated Chu-Liu-Edmonds algorithm \cite{china}-\cite{GGST} for outgoing trees with a given root \footnote{The Chu-Liu-Edmonds algorithm is originally formulated for outgoing trees. It is also called optimal branching. By reversing the arc directions, we obtain entering trees. }, which, in turn, is the only algorithm for constructing minimal directed trees.

Let us address the question of how to use the properties of the weights of the barrier graph $V$ (that is, use the corresponding potential graph {\sf P}) to construct minimal directed forests and trees.

\subsection{Construction of minimal trees of the barrier digraph}

Claims 4 and 5 allow us to easily construct minimal trees of an directed graph $V$ on an arbitrary subset ${\cal S}\subseteq{\cal N}$ (as long as it is connected), including a spanning minimal tree. Indeed, the Chu-Liu-Edmonds algorithm for constructing a spanning minimal tree with a given root has complexity $O(N^2)$ for dense directed graphs. If it is required to find the minimum without regard to the root, then in the general situation it must be applied $N$ times. This increases the complexity of the algorithm and it becomes $O(N^3)$. In our situation for the barrier orgraph it is sufficient to apply this algorithm exactly once. According to Claim 5, the root is immediately assigned to the vertex $l$ with the minimal loop: $p_{yy}=\min\limits_{t\in{\cal N}}p_{tt}$. In addition, in the resulting tree, by simply reassigning the root, minimal trees with all remaining roots are immediately determined. Their weights are determined according to Claim 4 from formula (\ref{lqq'}). The complexity of the algorithm remains $O(N^2)$.

Moreover, there is no need to use the Chu-Liu-Edmonds algorithm, which is quite non-trivial in implementation. We can move on to an undirected graph of potential {\sf P}. For it, use the simple algorithms of Prim \cite{Pr} or Kruskal \cite{Kru}. Formally, their complexity for dense graphs is $O(N^2)$. Using any of them, we obtain a minimal spanning undirected tree. Giving it an orientation by simply assigning a root, we again obtain minimal directed trees according to Theorem 3. All their weights are immediately found by Claim 5 from formula (\ref{ln}), which does not increase the overall complexity and it remains $O(N^2)$. If it is necessary to determine a minimal tree regardless of the root, then again the root is immediately assigned to the vertex $l$ with the minimal loop. Thus, the construction of a spanning minimum tree for a barrier digraph is implemented more simply and $N$ times more efficiently than for an arbitrary digraph. 

\subsection{The problem of constructing minimum spanning forests}

With spanning entering forests the situation is significantly different. Knowing the minimal undirected forest {\sf F} does not lead to obtaining the corresponding directed minimal forest. It is not difficult to determine the smallest value of $p_{yy}$ on each connected component, and then orient the forest. The resulting forest, unfortunately, is not minimal. The problem is that the connectivity components of undirected and directed minimal forests can be completely different. It is even possible that the vertex set of any tree of any forest from $\tilde{\textsc F}^k$ does not coincide with any vertex set of any tree of any directed forest from $\tilde{\cal F}^k$.

\subsection{Example} 
\begin{figure}[h]
\unitlength=1mm
\begin{center}
\begin{picture}(120,51)

\put(2,46){$a$}
\put(72,46){$a$}
\put(3,22){$a$}
\put(72,21){$a$}

\put(48,22){$c$}
\put(48,46){$c$}
\put(118,46){$c$}
\put(118,21){$c$}

\put(25,38){$b$}
\put(95,38){$b$}
\put(25,13){$b$}
\put(95,13){$b$}

\put(25,49){$V$}
\put(25,23){{\sf P}}

\put(17,42){4}
\put(83,38){1}
\put(13,37){1}

\put(33,42){3}
\put(37,37){2}

\put(15,13){5}
\put(35,13){4}
\put(25,2){1}
\put(5,12){4}
\put(46,12){2}
\put(106,13){4}

\put(5,20){$\bullet$}
\put(25,10){$\bullet$}
\put(45,20){$\bullet$}
\put(26,6){\circle{10}}
\put(6,16){\circle{10}}
\put(47,16){\circle{10}}
\put(7,21){\line(2,-1){20}}
\put(46,21){\line(-2,-1){20}}

\put(5,45){$\bullet$}
\put(25,35){$\bullet$}
\put(45,45){$\bullet$}
\put(6,45){\vector(2,-1){19}}
\put(26,37){\vector(-2,1){19}}
\put(46,45){\vector(-2,-1){19}}
\put(26,37){\vector(2,1){19}}

\put(90,48){$\{ F\}={\tilde{\cal F}^2}$}
\put(88,23){$\{ {\sf F}\}=\tilde{\textsc F}^2$}

\put(75,20){$\bullet$}
\put(95,10){$\bullet$}
\put(115,20){$\bullet$}

\put(116,21){\line(-2,-1){20}}

\put(75,45){$\bullet$}
\put(95,35){$\bullet$}
\put(115,45){$\bullet$}
\put(76,46){\vector(2,-1){19}}
\end{picture} 
\caption{\small Above is the barrier graph $V$ and its spanning minimum entering forest $F$, consisting of two trees; below is the corresponding potential graph {\sf P} and its spanning minimum forest {\sf F}, consisting of two trees.}
\label{examp}
\end{center}
\end{figure}
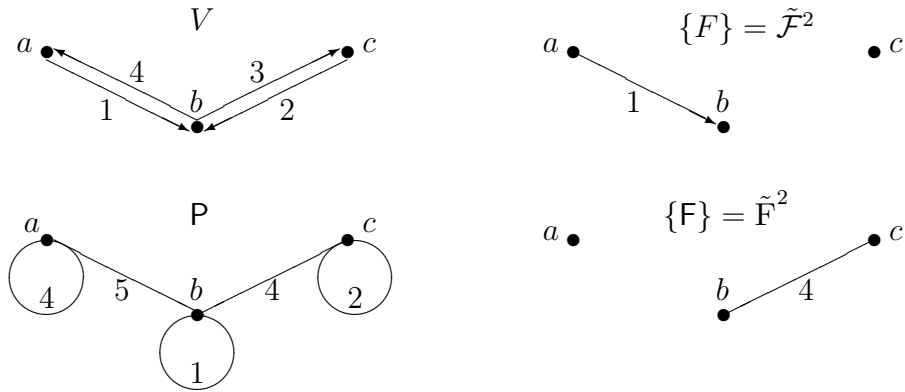
Let us consider an example (see Fig. \ref{examp}) demonstrating a significant difference between the minimal entering forests of the barrier graph $V$ and the minimal forests of the corresponding potential graph {\sf P}. Here the weights of the arcs $V$ are equal to: $v_{ab}=1$, $v_{ba}=4$, $v_{bc}=3$, $v_{cb}=2$. The potential graph {\sf P} has edge (and loop) weights: $p_{ab}=5$, $p_{bc}=4$, $p_{aa}=4$, $p_{bb}=1$, $p_{cc}=2$. As is easy to see, (\ref{poten}) holds for $V$ and {\sf P}.  The minimum spanning forest $F$ of a digraph $V$ consisting of two trees is unique. It consists of an empty tree rooted at $c$ and a tree containing a single arc $(a,b)$ with weight $1$. The minimum spanning forest {\sf F} consisting of two trees is also unique. But it contains an empty tree consisting of one vertex $a$ and a tree consisting of one edge $(b,c)$ of weight $4$. Thus, the vertex set of any tree in the entering forest $F$ does not coincide with any tree vertex set in the forest {\sf F}. 

This example demonstrates that knowing the minimal forests of the potential graph {\sf P} does not help in finding the minimal entering forests of the corresponding barrier digraph $V$. This is quite natural, since the forests of the graph {\sf P} do not contain loops. Moreover, each loop $(q,q)$ of the graph {\sf P} participates in the weight of all arcs outgoing from vertex $q$ in the directed graph $V$. Because of this, other sets of tree vertices are formed. In the case of spanning trees (or trees on a given subset of a set of vertices), this problem does not exist, since the set of tree vertices is known in advance. And on the selected subset, if we set the root, in the weight of any entering tree, according to (\ref{TqT}), the same sum of loops of the graph {\sf P} will appear. Therefore, the type of the minimal entering tree ceases to depend on loops.

\subsection{Guiding Considerations}

Let there be some minimal spanning entering forest $F\in\tilde{\cal F}^{k}$. Then the weights of its trees $T_l^F$, $l\in{\cal K}_F$ are known. According to Property 3, they are equal to $\Upsilon^{T^F_l}=\lambda_{{\cal V}T^F_l}^{\bullet l}=\lambda_{{\cal V}T^F_l}^\bullet$. By the way, for any set ${\cal S}\subseteq {\cal N}$, the value $\lambda_{\cal S}^{\bullet}$, according to Claim 5, is easily determined ((\ref{ln}), (\ref{lu})) from the potential graph {\sf P}. For it, the corresponding minimum $\nu_{\cal S}$ is effectively found using Prim's or Kruskal's algorithm on any set of vertices ${\cal S}$.

According to Theorem 2, to construct the minimal descendant of the entering forest $F$, it is necessary to know the values $\lambda_{{\cal V}T^F_l}^\circ$. 
Then it is necessary to determine the increments $\lambda_{{\cal V}T^F_l}^\circ-\lambda_{{\cal V}T^F_l}^\bullet$ (see (\ref{ll})) and find the smallest among them (let, for definiteness, the smallest increment at $l=y$).
After this, the arcs of the incoming tree $T^F_y$ are replaced by the arcs of the tree $T$ on which the minimum of $\lambda_{{\cal V}T^F_y}^\circ$ is achieved. If in the tree $T$ the arc that comes from ${\cal V}T^F_y$ (and it is the only one) is the arc $(a,b)$, then in fact in the tree $T^F_y$ the root is reassigned to the vertex $a$ and the arc $(a,b)$ is added. As a result, we obtain the minimal entering forest $G\in\tilde{\cal F}^{k-1}$, which is a descendant of the entering forest $F$.

According to Theorem 2 (\ref{potom}) and Claim 9 (\ref{ll}), the problem is to determine the minimum of the increment 

\begin{equation}
\min_{l\in{\cal K}_F}(\lambda_{{\cal V}T^F_l}^\circ -\lambda_{{\cal V}T^F_l}^{\bullet})=\min_{l\in{\cal K}_F}\bigg( \min\limits_{\begin{smallmatrix} q\in{\cal V}T^F_l, \\ r\notin{\cal V}T^F_l \end{smallmatrix}}p_{qr} -\min\limits_{t\in{\cal V}T^F_l}p_{tt}\bigg) \ . 
\label{jll}
\end{equation}

With the definition of values   
$\min\limits_{t\in{\cal V}T^F_l}p_{tt}$ there is no problem. 
They are equal to the weights of the loops $p_{ll}$ of the potential graph {\sf P} at the roots of the corresponding trees of the entering forest $F$. So they are known in advance. The expression (\ref{jll}) takes the form

\begin{equation}
\min_{l\in{\cal K}_F}(\lambda_{{\cal V}T^F_l}^\circ -\lambda_{{\cal V}T^F_l}^{\bullet})=\min_{l\in{\cal K}_F}\bigg( \min\limits_{\begin{smallmatrix} q\in{\cal V}T^F_l, \\ r\notin{\cal V}T^F_l \end{smallmatrix}}p_{qr} -p_{ll}\bigg) \ . 
\label{jlll}
\end{equation}
Thus, the problem is reduced to determining the minimal barriers between the set of vertices of a particular tree and the remaining vertices.

\subsection{Barriers between sets of tree vertices}

We transform the expression (\ref{jlll}) to introduce minimal barriers between different trees: 

\begin{equation}
\min_{l\in{\cal K}_F}\bigg(\min_{i\in{\cal K}_F\setminus\{ l\}} \min\limits_{\begin{smallmatrix} q\in{\cal V}T^F_l, \\ r\in{\cal V}T^F_{i} \end{smallmatrix}}p_{qr}-p_{ll}\bigg)=
\min_{\{l,i\}\subset{\cal K}_F} \bigg( \min\limits_{\begin{smallmatrix} q\in{\cal V}T^F_l, \\ r\in{\cal V}T^F_{i} \end{smallmatrix}}p_{qr}-p_{ll}\bigg)  \ . 
\label{jlll'}
\end{equation}
The value in brackets now corresponds to the potential barrier that must be overcome to get from the deepest point of the tree $T^F_l$, namely the vertex $l$, to the set of vertices of the tree rooted at $i$. And the question now is how to efficiently count transitions between trees.

Let $F\in\tilde{\cal F}^k$. We introduce an enlarged directed graph $V^{k}$ of barriers between the sets of vertices of the trees of the forest $F$. The set of vertices is the roots of the trees: ${\cal V}V^{k}={\cal K}_F$. We define the weights $v^k_{li}$ of the arcs as
\begin{equation}
v^k_{li}= \min\limits_{\begin{smallmatrix} q\in{\cal V}T^F_l, \\ r\in{\cal V}T^F_{i} \end{smallmatrix}}p_{qr}-p_{ll}  \ , \ \ l\neq i. 
\label{vk}
\end{equation}
If in the directed graph $V$ there are no arcs between the sets of vertices of the trees $T^F_l$ and $T^F_{i}$, then we assume that in the directed graph $V^k$ there are no arcs $(l,i)$ and $(i,l)$.  

Similarly, the enlarged graph of the potential ${\sf P}^k$ is defined, for which ${\cal V}{\sf P}^k={\cal V}V^k$. The weights of the loops remain the same $p^k_{ll}=p_{ll}$. The weights of the edges are equal to
\begin{equation}
p^k_{li}= \min\limits_{\begin{smallmatrix} q\in{\cal V}T^F_l, \\ r\in{\cal V}T^F_{i} \end{smallmatrix}}p_{qr}  \ , \ \ l\neq i. 
\label{pk}
\end{equation}
If the graph {\sf P} does not contain edges between the sets of vertices of the trees $T^F_l$ and $T^F_{i}$, then we assume that the enlarged graph ${\sf P}^k$ does not contain the edge $(l,i)$.
 
By construction, the introduced directed graph $V^k$ is a barrier directed graph, and ${\sf P}^k$ is the corresponding potential graph.

Let the arc of the minimum weight of the enlarged barrier graph $V^k$ be the arc $(y,x)$, and let $(a,b)$ be that edge of the original potential graph for which

\begin{equation}
p_{ab}= \min\limits_{\begin{smallmatrix} q\in{\cal V}T^F_x, \\ r\in{\cal V}T^F_{y} \end{smallmatrix}}p_{qr}=p^k_{yx} , \ \  a\in{\cal V}T^F_y, \ b\in{\cal V}T^F_x \ . 
\label{pab}
\end{equation}
Thus, 

\begin{equation}
\min_{(l, i)\in{\cal A}V^k}v^k_{li}=v^k_{yx}=p^k_{yx}-p^k_{yy}=p_{ab}-p_{yy} \ .
\label{vyx}
\end{equation}
In this case, $a\in{\cal V}T^F_y$, $b\in{\cal V}T^F_x$, and the value (\ref{vyx}) by construction coincides with both (\ref{jlll}) and (\ref{ll}), where ${\cal S}={\cal V}T^F_y$, and therefore with (\ref{potoml}). Thus, adding an arc $(a,b)$ to the forest $F$ and reorienting the tree $T^F_y$ by assigning the root to the vertex $a$, we obtain an encroaching forest $G\in\tilde{\cal F}^{k-1}\cap{\cal R}^F$. For it, according to Proposition 2 (\ref{pere}) and Propositions 8 (\ref{lt}) and 9 (\ref{lab}), the weight of the increased tree changes (the tree with the root at $x$ absorbed the tree with the root at $y$): 

\begin{equation*}
\lambda_{{\cal V}T^G_x}^\bullet=\lambda_{{\cal V}T^F_x}^\bullet+\lambda_{{\cal V}T^F_y}^\circ= \lambda_{{\cal V}T^F_x}^\bullet + \lambda_{{\cal V}T^F_y}^\bullet+ p_{ab}-p_{yy} \ .
\end{equation*}
Accordingly, the weight of the new minimum forest is equal to 

\begin{equation}
\Upsilon^G=\Upsilon^F+ p_{ab}-p_{yy} \ . 
\label{uaby}
\end{equation}

\subsection{Recalculation of barriers and potential}

The resulting digraph $G$ is minimal and consists of $k-1$ trees. Let us create an enlarged  digraph of barriers $V^{k-1}$ for it. 
In the digraph of barriers $V^{k-1}$ (and the potential graph ${\sf P}^{k-1}$), compared to $V^k$ (and ${\sf P}^k$), the vertex $y$ has disappeared from the set of vertices. The weights of arcs (edges) not incident to vertex $x$ have remained the same.

 \begin{equation}
v^{k-1}_{li}=v^k_{li} \ , p^{k-1}_{li}=p^k_{li} \ , \ \ \{ l,i\}\subset {\cal K}_G\setminus\{ x\} \ .
\end{equation}

The loop weights of the potential graph ${\sf P}^{k-1}$ for $l \neq x$ are obviously the same. The weight of the loop $(x,x)$ has not changed either, since this vertex remains in the graph.
\begin{equation}
p^{k-1}_{tt}=p^k_{tt} \ , \ \ t\in{\cal V}{\sf P}^{k-1}={\cal V}{\sf P}^{k}\setminus\{ y\} \ .
\end{equation}

For the weights of the edges of the barrier graph ${\sf P}^{k-1}$ incident to the vertex $x$, we have   

\begin{equation*}
p^{k-1}_{lx}= \min\limits_{\begin{smallmatrix} q\in{\cal V}T^G_l, \\ r\in{\cal V}G^F_{x} \end{smallmatrix}}p_{qr}=\min\limits_{\begin{smallmatrix} q\in{\cal V}T^G_l, \\ r\in{\cal V}T^F_{x}\cup{\cal V}T^F_y \end{smallmatrix}}p_{qr}=\min_{t\in\{ x,y\}} \min\limits_{\begin{smallmatrix} q\in{\cal V}T^F_l, \\ r\in{\cal V}F^F_{t} \end{smallmatrix}}p_{qr}  \ , \ \ l\in{\cal V}{\sf P}^{k-1}\setminus \{ x\} \ ,  
\end{equation*}
where do we get that

\begin{equation}
p^{k-1}_{lx}=\min(p^k_{lx}, p^k_{ly}) \ , \ \ l\neq x \ .
\label{plx}
\end{equation}
After this, the weights of the corresponding arcs of the barrier digraph $V^k$ are immediately determined

 \begin{equation}
v^{k-1}_{lx}=p^{k-1}_{lx} -p^{k-1}_{ll}  , \ \  v^{k-1}_{xl}=p^{k-1}_{lx} -p^{k-1}_{xx}  , \ \ l\in{\cal V}{\sf P}^{k-1}\setminus \{ x\} \ .
\label{vlx}
\end{equation}

\subsection{Algorithm for constructing minimum entering forests and its complexity}

The weight matrices are denoted as the (or)graph itself. In the absence of an arc (edge), the weight is interpreted as $\infty$.

Zero step (start of the algorithm).

At this step, we have weight matrices ${\bf V}^N={\bf V}$, and ${\bf P}^N={\bf P}$. The initial minimum forest is a spanning forest $F\in\tilde{\cal F}^N$, which is empty. It has no arcs. Its weight is $\Upsilon^F=0$.

Lists $\Omega_t$ of edges (the orientation within the list is unimportant, so we use the term edge rather than arc) of the trees of the minimal forest are created. These lists are indexed by the roots. Knowing the edges and roots, it is easy to determine what an entering forest looks like. At the initial stage, there are $N$ of these lists, and they are all empty.  

Next, the procedure is recurrent.

$l$-th step ($l=1,2,\ldots N-1$).

The currently constructed minimal entering forest $F$ contains $k=N+1-l$ roots. It is defined by the current lists $\Omega_t$ of edges and the roots $t\in{\cal K}_F$. The current matrix ${\bf V}^k$ of barriers and the matrix ${\bf P}^k$ of potentials are known.

$a)$ 
 The minimum element in the barrier matrix 
${\bf V}^k$ is determined. The roots corresponding to this element are given the labels $y$ and $x$: 
 
 $$v_{yx}=\min\limits_{(l,i)\in{\cal V}V^k}v_{li}^k \ .$$   

$b)$ The vertices of the edge of the original potential graph corresponding to the element $p_{yx}^k$ receive the labels $a$ and $b$:
$$
p_{ab}=p^k_{yx} \ , \ \ a\in{\cal V}T_y^F, \ \ b\in{\cal V}T_x^F \ . 
$$

$c)$ 
The tree $T^F_y$ is reoriented by assigning a root at $a$. An arc $(a,b)$ is added to it (with its weight $v_{ab}$). This turns the forest $F$ into the new minimum entering forest under study belonging to $\tilde{\cal F}^{k-1}$. The tree rooted at $x$ absorbs the tree rooted at $y$. We note the transformations necessary for this. According to (\ref{uaby}), the new weight of the minimum forest is

\begin{equation*}
\Upsilon^F:=\Upsilon^F+ p_{ab}-p_{yy} \ .
\end{equation*}
The edge lists $\Omega_x$ and $\Omega_y$ are combined into a single list $\Omega_x$ and the edge $(a,b)$ is added to it. This completely defines the new tree $T^F_x$ and the new minimal entering forest $F$ itself. 

$d)$ Matrices ${\bf V}^{k-1}$ and ${\bf P}^{k-1}$ are constructed. For the edge weight matrix, $k-2$ comparisons (\ref{plx}) are required to be performed to select the minimum ones. They are used to adjust the current edge matrix of the potential. 
At the same time, the edges of the original potential graph that yield these minima (which are used in point $b)$ are remembered. 
The modified $2(k-2)$ arcs (\ref{vlx}) are calculated and the matrix of arc weights is adjusted. Now the index corresponding to the vertex $y$ is removed from the matrices. It is no longer the root of the tree, and the dimension of the matrices is reduced by $1$ and becomes equal to $k-1$. The matrices 
${\bf V}^{k-1}$ and ${\bf P}^{k-1}$ are obtained.

$e)$ We set $l:=l+1$. If $l=N$, then end. Otherwise, go to point $a)$. 

Property 4, common to oriented graphs, and Theorem 2, taking into account the equality of the formula (\ref{potoml}) in it to the value (\ref{vyx}), which is fully used in the algorithm, guarantee that the proposed procedure can construct any minimal forest from $\tilde{\cal F}^k$, $k=1,2,\ldots , N$.

\subsection{Complexity of the algorithm}

When estimating the complexity, we assume that the original graph is dense $|{\cal A}V|\backsim N^2$. 

$a)$ 
Selecting the minimum element at each step with the current number of roots $k$ (the matrix size is $k\times k$) takes $O(k^2)$ operations. Summing up over all steps, we get the complexity of this most expensive item $O(N^3)$.

$d)$ Comparison and correction of arc weights requires $O(k)$ operations. Summing up over all steps, we obtain the complexity of this point $O(N^2)$.

The total complexity is therefore equal to $O(N^3)$.

\section{More about the properties of minimal forests of the barrier digraph}

According to Theorem 3, the set of minimal incoming forests created as a result of the algorithm has the following property by construction. Let $F\in\tilde{\cal F}^n$ and $F'\in\tilde{\cal F}^m$, $n>m$, and they are obtained by the above procedure. When the orientation is removed, they turn into some spanning forests {\sf F} and {\sf F'} of the potential graph {\sf P}, respectively (they are, generally speaking, no longer minimal, as the example considered shows, although each individual tree is minimal on its set of vertices). Then ${\cal E}{\sf F}\subset{\cal E}{\sf F'}$. An arbitrary directed graph $V$ does not have such a property for its minimal forests when removing the orientation (regardless of the weights themselves), regardless of the algorithm by which they were obtained.

If we abstract from the algorithm and talk about the entering forests themselves, then the described property can be formulated as follows. Let $n>m$, then for any $F\in\tilde{\cal F}^n$ there is an $F'\in\tilde{\cal F}^m$ (and vice versa, for any $F'\in\tilde{\cal F}^m$ there is an $F\in\tilde{\cal F}^n$) such that when removing the orientation for the resulting undirected forests {\sf P} and ${\sf P}'$, respectively, ${\cal E}{\sf F}\subset{\cal E}{\sf F'}$.  

\section{Road network for cities of varying importance}

To illustrate the usefulness of Prim and Kruskal's algorithms, an example is usually given with a network of roads that need to connect $N$ cities. The cost weight of a direct connection between city $i$ and city $j$ is known and equal to $p_{ij}$. An undirected graph {\sf P} with such edge weights is created and its minimum spanning tree is found. 
We will use the same example to demonstrate the significant advantages of minimum spanning directed forests.
 
The real situation is far from ideal and it may well be that funds can only be allocated to build $l$ roads. And it remains to be seen which cities will be connected in order to spend the minimum amount of funds. For now, this does not complicate the task much. The point is that the Kruskal algorithm procedure at each step builds a spanning forest of minimum weight, which completely solves the issue of building $l$ roads out of $N-1$ required. And when new funds appear, it is known where to continue construction.

But the reality is still more complicated.  The fact is that the cities themselves are obviously unequal in importance. Therefore, when assessing the effectiveness of construction costs, it is necessary to subtract some values of cities that are connected by roads from the total direct costs $p_{ij}$. If you immediately select the $k$ most important cities and connect only them, then this will obviously not be optimal. And this is not only because they may be too far from each other. If additional funds appear, then the connection with the cities further in importance will become less and less optimal. Let's say the city was located near an already built road between more important objects, and now it is necessary to build a road to it from one of them, instead of immediately passing through it during the previous construction. Moreover, it is essential that the sequence in which the construction is carried out is optimal. After all, funds may suddenly run out and they will not be enough for the planned $l$ roads, but we would like the costs for those already built to be effective.

The next natural idea is to subtract the value of the linked cities from the direct costs at each step in the utility assessment. This looks appealing, but the devil is in the details. Indeed, if you connect two cities, each of which is already connected to someone, then there is no need to subtract the value of these cities at all (it is already taken into account). But if you connect cities that have not yet been affected by the connection, then you need to subtract the value of both of them. And finally, if one of the cities is already connected to someone, and the second is not connected to anyone, then only the value of the latter should be subtracted. This indicates that with this approach, there is no single regular procedure by which the efficiency of costs at each step would be consistently studied.

And that's where minimal entering forests and algorithms for their construction come to the rescue. Indeed, the values of cities are declared by the weights of the loops $p_{ii}$. A graph of potential {\sf P} and a digraph of barriers $V$ with weights of arcs (\ref{poten}) is created (the weight of each arc $v_{ij}=p_{ij}-p_{ii}$  it makes sense for the effective cost of connecting the cities of $i$ and $j$, taking into account the value of the city of $i$). After which the algorithm above sequentially determines the minimum spanning entering forests of the barrier graph $V$. With this approach, at each step, when calculating the effective utility, the value of exactly one city out of two connected by the road will be taken into account.  Effective costs are minimal. Note that according to the specified procedure, if all $N$ cities are connected, then by virtue of Theorem 3, direct costs are also minimal.

\section{Entering and outgoing forests and trees}

Classically \cite{T} by directed trees it is customary to understand outgoing trees (or those growing from the root, as well as rooted). The same applies to algorithms. For example, all the main variants of the algorithm for the minimum spanning tree \cite{china}-\cite{GGST} build exactly an outgoing tree. However, when describing the minors of matrices, the first to emerge were exactly the opposite --- the entering forests \cite{Fied-Sed:58}. Then, in the spectra of directed graphs, namely, in the study of the Laplace matrix (a matrix with a zero sum over a row) of the directed graph: $l_{ij}=-v_{ij}$, $j\neq j$; $l_{ii}=-\sum\limits_{j\neq i}l_{ij}$, entering forests arise \cite{V}. Further, after the analysis of the pseudo-potential, entering forests themselves arise again in the lemmas on Markov chains in \cite{VF}. In all these works, emtering forests are called ${\cal W}$-graphs (${\cal W}$ is the set of vertices from which arcs do not outgo). When deriving the formula for all minors \cite{chai, moon}, the digraph is immediately called a forest, which in strict notation is an entering forest. For the standard Laplace matrix of an digraph, no preliminary actions are required, and in the study of its spectrum and eigenspaces, entering forests arise automatically \cite{V3}-\cite{V5}. 

Using out-forests in matrix analysis requires a bit more work. 
First, the arcs of the directed graph are reversed, and the Laplace matrix (with zero sum per row) is created for the resulting directed graph. In its analysis, but in the notation of the weights of the original directed graph (without reorienting the arcs), out-forests arise. This matrix itself, with this approach, is called the Kirchhoff matrix ($l_{ij}=-v_{ji}$, $j\neq j$; $l_{ii}=-\sum\limits_{j\neq i}l_{ij}$) \cite{ChA}. This method of transition somewhat confuses the situation.    

When studying the properties of an directed graph through the matrices corresponding to it, it would be worthwhile to use a unified approach to creating these matrices, without special tricks. It seems more natural to follow the historical approach and names \cite{CDZ}. Namely, a weighted adjacency matrix ${\bf V}$ of an directed graph $V$ without loops with elements $v_{ij}$ is created. Further, by analogy with undirected graphs, it is proposed to create two diagonal matrices of half-degrees (valences) of entry and exit ${\bf D}^{in}$ and ${\bf D}^{out}$ with diagonal elements $d^{out}_{ii}=\sum\limits_{j\neq i}v_{ij}$, $d^{in}_{ii}=\sum\limits_{j\neq i}v_{ji}$ (in undirected graphs this is one matrix of degrees). After this (also by analogy with undirected graphs), the matrices of incoming and outgoing conductivities are determined: ${\bf C}^{in,out}={\bf D}^{in,out}-{\bf V}$ (in the case of undirected graphs, this is the same conductivity matrix). With this approach, the Laplace matrix turns out to be the matrix of outgoing conductivities ${\bf C}^{out}$ (it is precisely the one whose sum of elements by rows is equal to zero) and its analysis is described by incoming forests. But the matrix of incoming conductivities ${\bf C}^{in}$ has a zero sum by columns, and outgoing forests arise in its analysis. This approach is completely symmetrical with respect to incoming and outgoing forests. Moreover, the weights of the arcs of the directed graph can be complex and this will not affect the formulas in any way. The matrix of incoming conductivities itself can naturally be considered a Kirchhoff matrix. There is no need to forcibly transpose it so that it also turns out to be a Laplace-type matrix.

Note also that the forest structure of an arbitrary directed graph may differ significantly for its incoming and outgoing forests. So the selection of incoming forests still suffers from some one-sidedness.  

For purely directed graph problems (not matrix type), replacing arcs with inverse ones does not change anything significant in principle. Therefore, the use of incoming or outgoing forest variants is more related to the specifics of a particular problem statement: is it more important to you from where or to? 

Of course, all the results of this paper and the papers \cite{V6}-\cite{V8} on which it is based can be reformulated for out-forests. This will entail some modification of definitions. Let's say that the weight of a subgraph of an directed graph $V$ on a subset ${\cal S}$ of a vertex set (\ref{ves}) should be defined not as the sum of the weights of the arcs outgoing from the vertices of the set ${\cal S}$, but as the sum of the weights of the arcs that enter ${\cal S}$. In general, the transition to outgoing forests is not particularly difficult, and the description in them is completely equivalent to the description in incoming forests. But in addition to the fact that the characteristics of the Laplace matrices are described by incoming forests, such forests are more natural in the energy interpretation, so they are used in this work.


\centerline{Abstract}

\begin{center}{Digraphs of potential barriers: properties of their tree structure and an algorithm for minimal forests}
\end{center}

\centerline{Buslov V.A.}

\parbox[t]{12cm}
{\small 
For a weighted digraph without loops $V$, the arc weights of which can be obtained from an undirected graph with loops ${\sf P}$ according to the rule $v_{ij}=p_{ij}-p_{ii}$, the properties are studied. An efficient algorithm for constructing directed trees of minimum weight and an efficient algorithm for constructing spanning directed forests of minimum weight are proposed.} 

\vspace{0.5cm}

St. Petersburg State University, Faculty of Physics, Department of Computational Physics

198504 St. Petersburg, Old Peterhof, st. Ulyanovskaya, 3

Email: abvabv@bk.ru, v.buslov@spbu.ru

\end{document}